\documentclass[11pt,reqno]{amsart} 

\usepackage{amssymb,latexsym}
\usepackage[draft=true]{hyperref}
\usepackage{multicol}
\usepackage{cite} 

\usepackage[height=190mm,width=130mm]{geometry} 

\theoremstyle{plain}
\newtheorem{theorem}{Theorem}
\newtheorem{lemma}{Lemma}

\theoremstyle{definition}
\newtheorem{definition}{Definition}

\theoremstyle{remark}
\newtheorem{remark}{Remark}
\newtheorem{example}{Example}

\numberwithin{equation}{section} 

\begin{document}
\title[Fixed function and its application]{Fixed function and its application to Medical Science} 

\author{Pooja Dhawan$^*$, Jatinderdeep Kaur}
\address{Thapar University\\ School of Mathematics\\ Patiala\\
  Punjab\\ India}

\email{pdhawan12@gmail.com, jkaur@thapar.edu}


\author{Vishal Gupta}

\address{Maharishi Markandeshwar University \\ Department of Mathematics\\ Mullana\\ Haryana
 \\ India}


\email{vishal.gmn@gmail.com}

\begin{abstract}
In the present paper, the concept of contraction has been extended in a refined manner by introducing $\mathfrak{D}$-Contraction defined on a family $\mathfrak{F}$ of bounded functions. Also, a new notion of fixed function has been introduced for a metric space. Some fixed function theorems along
with illustrative examples have also been given to verify the effectiveness of our results. In addition, an application to medical science has also been presented. This application is based on best approximation of treatment plan for tumor patients getting intensity
modulated radiation therapy(IMRT). In this technique, a proper DDC matrix truncation has been used that significantly improves accuracy
of results. In 2013, Z. Tian \textit{et al.} presented a fluence map optimization(FMO) model for dose calculation by splitting the DDC
matrix into two components on the basis of a threshold intensity value. Following this concept, a sequence of functions
can be constructed through the presented results which contains different dose distributions corresponding to different patients and finally it converges to a fixed function. The fixed function obtained in this application represents the suitable doses of a number of tumor patients at the same
time. A nice explanation has also been given to check the authenticity of the application and the results.
\end{abstract}



\maketitle

\section{Introduction and preliminaries}
The concept of $\textit{``Contraction"}$ for a metric space was firstly introduced by polish mathematician Stefan Banach \cite{ban} to prove the existence and uniqueness of a fixed point. His principle known as \textit{``Banach Contraction"} ensures that the application of a continuous self mapping on two points of a complete metric space contracts the distance between those two points. According to his result, \textit{``A contraction self mapping defined on} \textit{a complete metric space possesses a unique fixed point which can be obtained as the limit of an } \textit{iteration scheme constructed by applying repeated images of the mapping} \textit{(starting from an} \textit{arbitrary point of space)"}. After that, many authors including Kannan \cite{kan}, Chatterjea \cite{cha}, $\acute{C}$iri$\acute{c}$ \cite{cir} gave extensions to this result by presenting  more robust contractive conditions.\\

By now, there exists considerable literature on all these generalizations in various spaces which are applicable in numerous fields. For more details, \{\cite{agg}, \cite{bha}, \cite{bor}, \cite{chan}, \cite{rho}, \cite{sha}, \cite{har}\} can be cited.\\

This paper deals with a unique approach in the field of contraction mappings introduced with a family of bounded functions. The contents of this paper have been divided into four sections. Section 1 is concerned with some basic definitions and results related to this paper. In section 2, main results have been presented with some illustrative examples whereas section 3 deals with an application to medical science. The last section of this paper presents the conclusion.\\

In order to prove the main results, we need some basic concepts, definitions and results from the literature.

\begin{definition}{\cite{ban}}
For a metric space $(X,d)$, a mapping $T:X\rightarrow X$ is called a contraction mapping on $X$ if for any real number $\lambda$ with $0\leq\lambda<1$, the following inequality holds:
\begin{equation*}
d(Tx,Ty)\leq\lambda d(x,y)~~for~~all ~~x,y \in X.
\end{equation*}
\end{definition}

\begin{remark}
 It can be easily seen that the distance between the images of any two points of a given set is contracting by a uniform
factor $\lambda<1$.
\end{remark}

\begin{example} {\cite{ban}}
Let $X=\mathbb{R}^2$ be a set equipped with standard metric $d$ (\textit{i.e.} {$d((x_1, y_1),(x_2, y_2))=\sqrt{(x_1-x_2)^2+ (y_1-y_2)^2}~for~all~~x_1,x_2,y_1,y_2 \in X$}) and $T:\mathbb{R}^2\rightarrow\mathbb{R}^2$ be the mapping defined as $Tx= \frac{3}{8}x$ for all $x\in\mathbb{R}^2$ where $x=(x_1,x_2)$. Then $T$ is a contraction on $X$ as $d(Tx,Ty)=\frac{3}{8}\sqrt{(x_1-y_1)^2+(x_2-y_2)^2}=\frac{3}{8}d(x,y).$
\end{example}

\begin{theorem} {\cite{ban}}
Let $(X,d)$ be a complete metric space and $T$ be the contraction mapping defined on $X$. Then $T$ possesses a unique fixed point $x$ in $X$ \textit{i.e.} $Tx=x$.
\end{theorem}

After this well known result, Reich{\cite{rei}} presented the following theorem:

 \begin{theorem} {\cite{rei}}
 Let $(X,d)$ be a complete metric space and $T$ be the self mapping defined on $X$ which satisfy the condition
\begin{equation*}
d(Tx,Ty)\leq \alpha d(x,Tx)+\beta d(y,Ty)+ \gamma d(x,y)
\end{equation*}
for all $x,y\in X$ and $\alpha,\beta,\gamma$ non negative with $\alpha+\beta+\gamma < 1$. Then $T$ admits a unique fixed point in $X$.
\end{theorem}

In 2012, Samet \textit{et. al.}{\cite{sam}} obtained some fixed point results by defining $\alpha-\psi$ contractive mapping as follows:

\begin{definition}{\cite{sam}}
 Let $\Psi$ be the family of all functions $\psi: [0,+\infty)\rightarrow[0,+\infty)$ satisfying the following properties:
\begin{eqnarray*}
&(1)&{\sum^{+\infty}_{n=1}}\psi^n(t)<+\infty~for~every~~t>0,~~where~\psi^n~is~~n^{th}~iterate~of~\psi ~~;\hspace{4cm}\\
&(2)&\psi~~is~~nondecreasing.
\end{eqnarray*}
\end{definition}

\begin{lemma}{\cite{sam}}
 If $\psi: [0,+\infty)\rightarrow[0,+\infty)$ is a nondecreasing function, then for each $t> 0, ~~\lim_{n\rightarrow +\infty}{\psi^n(t)}=0$ implies $\psi(t)< t$.
\end{lemma}

\begin{lemma}{\cite{sam}}
 If $\psi\in \Psi$, then the function $\psi$ is continuous at $0$.
 \end{lemma}

\begin{definition}{\cite{sam}}
Let $(X,d)$ be a metric space and $T$ be a given self mapping defined on $X$. The mapping $T$ is said to be an $\alpha-\psi$ contractive mapping if there exists two functions $\alpha: X\times X\rightarrow [0,+\infty)$ and $\psi\in \Psi$ satisfying
\begin{eqnarray*}
\alpha(x,y) d(Tx,Ty)\leq \psi(d(x,y))~~for~~all~~x,y \in X.
\end{eqnarray*}
\end{definition}

\begin{definition}{\cite{sam}}
Let $T: X\rightarrow X$ and $\alpha: X\times X\rightarrow [0,+\infty)$. The mapping $T$ is known as $\alpha$-admissible mapping if
\begin{eqnarray*}
\alpha(x,y)\geq 1 \Rightarrow \alpha(Tx,Ty)\geq 1~for~~every~~x,y \in X.\\
\end{eqnarray*}
\end{definition}

\section{Main results}

This section presents some fixed function theorems using the notions of fixed function and $\mathfrak{D}$-Contraction.

 \begin{definition}
\textbf{Fixed function}: Let $\mathfrak{D}$ be any self mapping defined on a family of functions $\mathfrak{F}$, then $f\in \mathfrak{F}$ is said to be fixed function of $\mathfrak{D}$ if $\mathfrak{D}f=f$.
\end{definition}

\begin{example}
Let $U=[1,2]$ and the mapping $\mathfrak{D}$ be defined as $\mathfrak{D}f(u)=f^2(u)-2f(u)+2$ for all $f\in \mathfrak{F}$ and $u\in U$. Then $f(u)=2$ for all $u\in U$ and $f(u)=1$ for all $u\in U$ are two fixed functions of $\mathfrak{D}$.
\end{example}

\begin{example}
Let $U=\mathbb{R}^+$(set of positive real numbers) and $\mathfrak{D}$ be the self mapping on $\mathfrak{F}$. Let $f\in \mathfrak{F}$ be a function defined on $U$ as
\begin{equation*}
f(u)=\begin{cases}-1&~~~0\leq u\leq 1 \\
~~0&~~~otherwise.\end{cases}
\end{equation*}
Then $f^3$ is a fixed function of $\mathfrak{D}$.
\end{example}

\begin{definition}
Let $(U,\hat{d})$ be a complete metric space and let $\mathfrak{F}$ be the collection of all bounded functions defined on $U$. Let $\mathfrak{D}$ be any self mapping on $\mathfrak{F}$. Then the given mapping is called $\mathfrak{D}$-contraction mapping on $\mathfrak{F}$, if for any real number $\lambda\in[0,1)$, we have
\begin{equation*}
d^*(\mathfrak{D}f,\mathfrak{D}g)\leq\lambda
d^*(f,g)~~for~~all~~f,g\in\mathfrak{F}
\end{equation*}
 where
\begin{equation}\label{e22}
 d^*(f,g)= sup \{\hat{d}(f(u),g(v))|~u,v\in U\}= sup \{|f(u)-g(v)|~|u,v\in U\}.
 \end{equation}
\end{definition}

\begin{remark}
Clearly, $d^*$ is  a metric on $\mathfrak{F}$ as $d^*(f,g)=0\Leftrightarrow f \sim g$ for all $f,g\in\mathfrak{F}$. Also, for all $u,v \in U$ and $f,g,h \in \mathfrak{F}$,
\begin{eqnarray*}
|f(u)-g(v)|&\leq&|f(u)-h(w)|+|h(w)-g(v)|\hspace{3cm}\\
           &\leq& \sup\{|f(u)-h(w)|~|u,w\in U\}\\
           &&+\sup\{|h(w)-g(v)|~|w,v\in U\}\\
\Rightarrow \sup\{|f(u)-g(v)|~|u,v\in U\}&\leq& \sup\{|f(u)-h(w)|~|u,w\in U\}\\
               &&+\sup\{|h(w)-g(v)|~|w,v\in U\}\\
\Rightarrow d^*(f,g)&\leq& d^*(f,h)+d^*(h,g).\\
\end{eqnarray*}
\end{remark}

Now we prove the main result.

 \begin{theorem}\label{t26}
 Let $(U,\hat{d})$ be a complete metric space with metric $\hat{d}$ defined as $\hat{d}(u,v)=|u-v|$ for all $u,v \in U$. Let $\mathfrak{F}$ be the collection of all bounded functions $f$ defined on $U$ with metric $d^*$(as defined in (\ref{e22})).\\
Also, let $\mathfrak{D}$ be the $\mathfrak{D}$-contraction mapping defined on $\mathfrak{F}$. Then there exists a unique fixed function $f\in\mathfrak{F}$ $\textit{i.e.}$ there exists some $f\in\mathfrak{F}$ such that $\mathfrak{D}f=f$.\\
\end{theorem}

\begin{proof} Let $f,g$ be any two functions from the family $\mathfrak{F}$. Since $\mathfrak{D}$ is the $\mathfrak{D}$-contraction mapping on $\mathfrak{F}$, therefore, there exists a real number $\lambda\in[0,1)$ such that
\begin{equation*}
d^*(\mathfrak{D}f,\mathfrak{D}g)\leq\lambda d^*(f,g)~~ for~~ all~~
f,g\in\mathfrak{F}\hspace{1cm}
\end{equation*}
where
\begin{equation*}
 d^*(f,g)= sup \{\hat{d}(f(u),g(v))|~u,v\in U\}.\hspace{0.8cm}
\end{equation*}
This further implies that
\begin{eqnarray*}
 d^*(\mathfrak{D^2}f,\mathfrak{D^2}g)&\leq&\lambda d^*(\mathfrak{D}f,\mathfrak{D}g)\\
                                      &\leq&\lambda^2 d^*(f,g)~~ for~~ all~~ f,g\in\mathfrak{F}.
\end{eqnarray*}
Continuing in the same manner, we get\\
\begin{equation}\label{e23}
d^*(\mathfrak{D^n}f,\mathfrak{D^n}g)\leq\lambda^n d^*(f,g)~ for~~
all ~~f,g\in\mathfrak{F}.
\end{equation}
{\textit{\textbf{Step I}: We will show that $\{f_{n}\}_{(n\in \mathbb{N})}$ is a cauchy sequence.}}\\
 Let $f_0$ be any function in $\mathfrak{F}$. Let us define the sequence $\{f_{n}\}_{(n\in \mathbb{N})}$ by setting
\begin{equation*}
f_{1}=\mathfrak{D}(f_{0}),
\end{equation*}
\begin{equation*}
f_{2}=\mathfrak{D}(f_{1})=\mathfrak{D^2}(f_{0}),
\end{equation*}
\begin{eqnarray*}
&&.~~~~.~~~~.~~\\
&&.~~~.~~~~.
\end{eqnarray*}
\begin{equation*}
f_{n}=\mathfrak{D}(f_{n-1})=\mathfrak{D}^{2}(f_{n-2})=...=\mathfrak{D^n}(f_{0}).
\end{equation*}
Let $p,q\in N$ be some positive integers with $p>q$. Let $p=q+t$ where $t$ is a +ve integer greater than equal to 1.
\begin{eqnarray*}
~~d^*(f_{q},f_{p})&=& d^*(f_{q},f_{q+t})\\
 &\leq& d^*(f_{q},f_{q+1})+d^*(f_{q+1},f_{q+2})+...+d^*(f_{q+t-1},f_{q+t})\\
 &=&d^*(\mathfrak{D^q}f_{0},\mathfrak{D^q}f_{1})+d^*(\mathfrak{D^{q+1}}f_{0},\mathfrak{D^{q+1}}f_{1})+...\\
 &&+d^*(\mathfrak{D^{q+t-1}}f_{0},\mathfrak{D^{q+t-1}}f_{1})\\
  &\leq& \lambda^q d^*(f_{0},f_{1})+\lambda^{q+1} d^*(f_{0},f_{1})+....\\
  &&+\lambda^{q+t-1} d^*(f_{0},f_{1})~~~\hspace{4cm}(using (\ref{e23}))\\
   &=&\lambda^q d^*(f_{0},f_{1}).[1+\lambda+\lambda^2+...+\lambda^{t-1}]\\
   &\leq& \frac{\lambda^q}{1-\lambda}d^*(f_{0},f_{1})~~where ~~\lambda<1.\\
\end{eqnarray*}
 Since $\mathfrak{F}$ is a family of bounded functions, therefore
\begin{equation*}
d^*(f_{q},f_{p})\rightarrow 0 ~~as~~~ p~,~q\rightarrow\infty.
\end{equation*}
Hence, $\{f_{n}\}_{(n\in \mathbb{N})}$ is a Cauchy sequence in $\mathfrak{F}$.\\

{\textit{\textbf{Step II}: Existence of fixed function}}.\\
As $\mathfrak{F}$ is the family of bounded functions defined on complete metric space $(U,\hat{d})$, therefore, $(\mathfrak{F},d^*)$ is a complete metric space and thus the sequence $\{f_{n}\}_{(n\in \mathbb{N})}$ is convergent in $\mathfrak{F}$.\\

Let $f\in\mathfrak{F}$ be the limit of $\{f_{n}\}_{(n\in \mathbb{N})}$ \textit{i.e} $\lim_{n\rightarrow\infty}~~f_{n}=f$.\\
  By the continuity of $\mathfrak{D}$, we get
\begin{equation*}
 \lim_{n\rightarrow\infty}~~ \mathfrak{D}f_{n}=\mathfrak{D}f.
\end{equation*}
Also,
\begin{equation*}
\mathfrak{D}f_{n}=f_{n+1}\rightarrow f~~ as
~~n\rightarrow\infty.\hspace{1cm}
\end{equation*}
Thus, uniqueness of limit implies that $\mathfrak{D}f=f$. This shows that $f$ is a fixed function of $\mathfrak{D}$.\\

{\textit{\textbf{Step III}: Uniqueness of fixed function.}}\\
Let $g$ be another fixed function of $\mathfrak{D}$ $\textit{i.e}$~~ $\mathfrak{D}g=g$ and $f\nsim g$.
\begin{eqnarray*}
 0\leq d^*(f,g)&=&d^*(\mathfrak{D}f,\mathfrak{D}g)\\
                   &\leq& \lambda d^*(f,g)\\
                 &<&d^*(f,g).
\end{eqnarray*}
 Thus, we arrive at a contradiction. Hence, $f$ is a unique fixed function of $\mathfrak{D}$.\\
 \end{proof}

\begin{example}
Let $U=\mathbb{R}$(set of real numbers) and $\hat{d}$ be the metric defined on $\mathbb{R}$. Clearly, $(U, \hat{d})$ is a complete metric space. Let $\mathfrak{F}$ be the family of bounded functions defined on $U$ and $d^*$ be the metric on $\mathfrak{F}$ defined as
 $$
 d^*(f,g)= sup \{\hat{d}(f(u),g(v))|~u,v\in U\}= sup \{|f(u)-g(v)|~|u,v\in U\}.
 $$
 It can be easily seen that $(\mathfrak{F}, d^*)$ is a complete metric space being the family of bounded functions defined on complete metric space $(U,\hat{d})$.\\

 Let
\begin{eqnarray*}
 f(u)&=&\begin{cases}1&~~u~is~~rational \\0&~~~u~~is~~irrational\end{cases}
 \end{eqnarray*}
and
\begin{eqnarray*}
g(u)&=&\begin{cases}-1&~~u~is~~rational
\\~~~0&~~~u~~is~~irrational.\end{cases}
\end{eqnarray*}
Let the mapping $\mathfrak{D}$ be defined as $\mathfrak{D}f= f^2$ for all $f\in \mathfrak{F}$.\\
Then, we only need to show that the mapping $\mathfrak{D}$ is a $\mathfrak{D}$-contraction mapping.\\
For this, we have
\begin{eqnarray*}
d^*(\mathfrak{D}f,\mathfrak{D}g)&=&sup\{~\hat{d}(\mathfrak{D}f(u),\mathfrak{D}g(v))~|u,v\in U\}\\
                                &=&sup\{|f^2(u)-g^2(v)|~|u,v\in U\}\\
&\leq& \lambda ~~sup\{|f(u)-g(v)|~|u,v\in U\}~~where~~0\leq \lambda < 1\\
\Rightarrow d^*(\mathfrak{D}f,\mathfrak{D}g)&\leq&\lambda d^*(f,g).
\end{eqnarray*}
Since all the conditions required for Theorem \ref{t26}  are fulfilled, therefore, there exists a unique fixed function of $\mathfrak{D}$. In this example, $f^2,~f^4,~f^6$ etc. yield same fixed function of
$\mathfrak{D}$.\\
\end{example}

\begin{example}
Let $U=[0,1]$ and $\hat{d}$ be the metric defined on $U$. Let $\mathfrak{F}=C[0,1]$ (\textit{i.e} set of all real valued continuous functions defined on [0,1]) and the mapping $\mathfrak{D}:\mathfrak{F}\rightarrow\mathfrak{F}$ be defined as
\begin{equation*}
\mathfrak{D}f(u)=\frac{2}{3}f(u) ~~~for~~~ all~~~ f\in\mathfrak{F}~~
and~~ u\in[0,1].
\end{equation*}
Here, $(U,\hat{d})$  is a complete metric space and $\mathfrak{F}=C[0,1]$ is the collection of all real valued continuous(and hence bounded) functions defined on $U=[0,1]$.\\

Let $f_{n}(u)=\frac{u^n}{n}$ for all $u\in[0,1]$.\\

Then $\{f_{n}(u)\}_{\left(u\in[0,1]\right)}$ is a uniformly convergent sequence in $\mathfrak{F}$ and therefore is a cauchy sequence.\\
Also, the given mapping is a $\mathfrak{D}$-contraction mapping as
\begin{eqnarray*}
d^*(\mathfrak{D}f,\mathfrak{D}g)&=&d^*\left(\frac{2}{3}f,\frac{2}{3}g\right)\\
                               &=&sup \left\{\hat{d}\left(\frac{2}{3}f(u),\frac{2}{3}g(v)\right)~|u,v\in U \right\}\\
                                &=& \frac{2}{3} sup\{~\hat{d}(f(u),g(v))~|u,v\in U\}\\
                               &<& \lambda d^*(f,g)~~for~~\frac{2}{3}< \lambda < 1.
\end{eqnarray*}
Since all the conditions required for Theorem \ref{t26} are fulfilled, therefore, there exists a unique fixed function of $\mathfrak{D}$. In this example, null function is a unique fixed function.\\
\end{example}

\begin{theorem}
 Let $(U,\hat{d})$ be a complete metric space (where $\hat{d}$ is the metric as defined earlier) and $\mathfrak{F}$ be the collection of all bounded functions $f$ defined on $U$ with metric $d^*$(as defined in (\ref{e22})).\\
Also, let $\mathfrak{D}$ be the modified $\mathfrak{D}$-contraction mapping on $\mathfrak{F}$ satisfying
\begin{equation*}
d^*(\mathfrak{D}f,\mathfrak{D}g)\leq \alpha d^*(f,\mathfrak{D}f)+
\beta d^*(g,\mathfrak{D}g)+\gamma d^*(f,g)
\end{equation*}
for all $f,g\in \mathfrak{F}$; $\alpha,\beta,\gamma$ non negative with $\alpha+\beta+\gamma<1$. Then  $\mathfrak{D}$ has a unique fixed function.
\end{theorem}

\begin{proof}
Let us define a sequence $\{f_{n}\}_{(n\in \mathbb{N})}$ of functions of $\mathfrak{F}$ in the following way:\\
Let $f_0\in \mathfrak{F}$ be any arbitrary function and $f_n=\mathfrak{D}f_{n-1}=\mathfrak{D}^n{f_0}$.\\

{\textit{\textbf{Step I}: $\{f_{n}\}_{(n\in \mathbb{N})}$ is a cauchy sequence in $\mathfrak{F}$.}}
\begin{eqnarray*}
d^*(f_1,f_2)&=&d^*(\mathfrak{D}f_0,\mathfrak{D}f_1)\\
                  &\leq& \alpha d^*(f_0,\mathfrak{D}f_0)+ \beta d^*(f_1,\mathfrak{D}f_1)+\gamma d^*(f_0,f_1)\\
                  &=&\alpha d^*(f_0,f_1)+ \beta d^*(f_1,f_2)+\gamma d^*(f_0,f_1)\\
                  &=&(\alpha+\gamma)d^*(f_0,f_1)+\beta d^*(f_1,f_2)
\end{eqnarray*}
\begin{eqnarray*}
&\Rightarrow& (1-\beta)d^*(f_1,f_2)\leq (\alpha+\gamma)d^*(f_0,f_1)\hspace{4cm}\\
&\Rightarrow& d^*(f_1,f_2)\leq
\left(\frac{\alpha+\gamma}{1-\beta}\right)d^*(f_0,f_1)\hspace{2cm}(where~~~\beta< 1)
\end{eqnarray*}
Similarly
\begin{eqnarray*}
d^*(f_2,f_3)&\leq& \left(\frac{\alpha+\gamma}{1-\beta}\right)d^*(f_1,f_2)\hspace{3cm}\\
                  &\leq& \left(\frac{\alpha+\gamma}{1-\beta}\right)^{2}d^*(f_0,f_1)
\end{eqnarray*}
and so on.\\

As $\left(\frac{\alpha+\gamma}{1-\beta}\right)<1$ and $f_0,f_1\in \mathfrak{F}$ are bounded, therefore, $\{f_{n}\}_{(n\in \mathbb{N})}$ is a cauchy sequence in $\mathfrak{F}$.\\
Since $\mathfrak{F}$ is complete being the family of bounded functions defined on complete metric space $(U,\hat{d})$, therefore, the sequence $\{f_{n}\}_{(n\in \mathbb{N})}$ is convergent in $\mathfrak{F}$(say it converges to $f\in \mathfrak{F}$).\\

{\textit{\textbf{Step II}: Existence of fixed function.}}\\
Now it will be shown that $f$ is a fixed function of $\mathfrak{D}$.
Let $s$ be any arbitrary +ve integer.
\begin{eqnarray*}
 d^*(f,\mathfrak{D}f)&\leq& d^*(f,f_s)+d^*(f_s,\mathfrak{D}f)\hspace{5cm}\\
                     &=& d^*(f,f_s)+d^*(\mathfrak{D}f_{s-1},\mathfrak{D}f)\\
                     &=& d^*(f,f_s)+d^*(\mathfrak{D}f,\mathfrak{D}f_{s-1})
\end{eqnarray*}
\begin{equation*}
\Rightarrow d^*(f,\mathfrak{D}f)\leq d^*(f,f_s)+\alpha
d^*(f,\mathfrak{D}f)+\beta d^*(f_{s-1},\mathfrak{D}f_{s-1})+\gamma
d^*(f,f_{s-1})
\end{equation*}
\begin{equation*}
\Rightarrow (1-\alpha)d^*(f,\mathfrak{D}f)\leq d^*(f,f_s)+\beta
d^*(f_{s-1},\mathfrak{D}f_{s-1})+\gamma d^*(f,f_{s-1})\hspace{1.1cm}
\end{equation*}
The right side expression can be made arbitrarily small enough by taking $s$ sufficiently large. Thus
\begin{eqnarray*}
 0~~\leq~~ d^*(f,\mathfrak{D}f)&<& \epsilon\\
\Rightarrow
d^*(f,\mathfrak{D}f)&=&0~~\hspace{0.2cm}\textit{i.e}~~f~is~a~fixed~function~of~\mathfrak{D}.
\end{eqnarray*}

{\textit{\textbf{Step III}: Uniqueness of fixed function.}}\\
Suppose $g\in \mathfrak{F}$ be another fixed function of $\mathfrak{D}$  \textit{i.e} $\mathfrak{D}g=g$ and $g\nsim f$.\\
Then
\begin{eqnarray*}
d^*(f,g)&=&d^*(\mathfrak{D}f,\mathfrak{D}g)\hspace{2cm}\\
         &\leq& \alpha d^*(f,\mathfrak{D}f)+\beta d^*(g,\mathfrak{D}g)+\gamma d^*(f,g)\\
\Rightarrow (1-\gamma)d^*(f,g)&\leq&0~~~~~~~(where~~\gamma< 1)\\
\Rightarrow d^*(f,g)&\leq&0
\end{eqnarray*}
which is a contradiction to our assumption. This implies that $f$ is unique.\\
\end{proof}

In this paper, we have extended the concept of $\alpha-\psi$ contractive mapping in the following manner:\\
\begin{definition}
 The mapping $\mathfrak{D}: \mathfrak{F}\rightarrow \mathfrak{F}$ is said to be an $\alpha-\psi$ contractive mapping if there exists two functions $\alpha: U\times U\rightarrow [0,+\infty)$ and $\psi\in \Psi$ satisfying
\begin{equation}\label{e24}
\alpha(f(u),g(v)) d^*(\mathfrak{D}f,\mathfrak{D}g)\leq
\psi(d^*(f,g))
\end{equation}
for all $f,g \in \mathfrak{F}$ and $u,v\in U$.
\end{definition}

\begin{definition}
Let $\mathfrak{D}: \mathfrak{F}\rightarrow \mathfrak{F}$ and $\alpha: U\times U\rightarrow [0,+\infty)$. The mapping $\mathfrak{D}$ is called an $\alpha$-admissible mapping if
\begin{equation*}
\alpha(f(u),g(v))\geq 1 \Rightarrow
\alpha(\mathfrak{D}f(u),\mathfrak{D}g(v))\geq 1
\end{equation*}
for every $f,g \in \mathfrak{F}$ and $u,v\in U$.\\
\end{definition}

\begin{theorem}\label{t3}
 Let $(U,\hat{d})$ be a complete metric space and $\mathfrak{F}$ be the collection of all bounded functions $f$ (defined on $U$) with metric $d^*$ (as defined in (\ref{e22})). Let $\mathfrak{D}:\mathfrak{F}\rightarrow \mathfrak{F}$ be an $\alpha-\psi$ contractive mapping. Also, suppose that
\begin{eqnarray*}
&(i)& \mathfrak{D}~~is~~{\alpha}-admissible.\hspace{10cm}\\
&(ii)& there~is~some~f_0\in \mathfrak{F}~for~which~~\alpha(f_0(u),\mathfrak{D}f_0(v))\geq 1~for~all~u,v\in U.\\
&(iii)& \mathfrak{D}~~is~~continuous.
\end{eqnarray*}
Then  $\mathfrak{D}$ possesses a fixed function in $\mathfrak{F}$.
\end{theorem}

\begin{proof}
 Let $f_0\in \mathfrak{F}$ be a function such that
 \begin{equation*}
\alpha(f_0(u),\mathfrak{D}f_0(v))\geq 1~~ for~~all~~u,v \in U.
\end{equation*}
Define the sequence $\{f_{n}\}_{n\in \mathbb{N}}$ in $\mathfrak{F}$ by $f_{n+1}= \mathfrak{D}f_n$ for every $n\in \mathbb{N}$. If $f_n=f_{n+1}$ for some $n\in \mathbb{N}$, then $f_n$ is a fixed function of $\mathfrak{D}$. Let us assume that $f_n\neq f_{n+1}$ for every $n\in \mathbb{N}$.\\

As by condition $(i)$, $\mathfrak{D}$ is  ${\alpha}$-admissible, therefore for all $u,v\in U$, we have
\begin{eqnarray*}
\alpha(f_0(u),f_1(v))&=& \alpha(f_0(u),\mathfrak{D}f_0(v))\geq 1\hspace{2cm}\\
\Rightarrow\alpha(\mathfrak{D}f_0(u),\mathfrak{D}f_1(v))&=&
\alpha(f_1(u),f_2(v))\geq 1
\end{eqnarray*}
By mathematical induction, we get
\begin{equation}\label{e25}
\alpha(f_n(u),f_{n+1}(v))\geq 1~~for~~all~~n\in \mathbb{N}~~and~~u,v
\in U.
\end{equation}
Using (\ref{e24}) and (\ref{e25}),
\begin{eqnarray*}
d^*(f_n,f_{n+1})&=&d^*(\mathfrak{D}f_{n-1},\mathfrak{D}f_n)\\
  &\leq&\alpha(f_{n-1}(u),f_n(v))d^*(\mathfrak{D}f_{n-1},\mathfrak{D}f_n)\\
  &\leq&\psi(d^*(f_{n-1},f_n))
\end{eqnarray*}
Repetition of above process implies
\begin{eqnarray*}
d^*(f_n,f_{n+1})\leq \psi^n(d^*(f_0,f_1))~~for~~all~~n\in
\mathbb{N}.
\end{eqnarray*}
Let $n>m\geq N$ for $N \in \mathbb{N}$. Using triangular inequality,
we have
\begin{eqnarray*}
d^*(f_m,f_n)&\leq& d^*(f_m,f_{m+1})+d^*(f_{m+1},f_{m+2})+d^*(f_{m+2},f_{m+3})+\\
&&...+d^*(f_{n-1},f_n)\\
&\leq&\psi^m(d^*(f_0,f_1))+\psi^{m+1}(d^*(f_0,f_1))+...+\psi^{n-1} (d^*(f_0,f_1))\\
&=&\sum^{n-1}_{k=m}\psi^k(d^*(f_0,f_1)).
\end{eqnarray*}
As ${\sum^{+\infty}_{n=1}}\psi^n(u)<+\infty$ for each $u> 0$, so $\{f_{n}\}_{n\in \mathbb{N}}$ is a Cauchy sequence in $\mathfrak{F}$ and being collection of bounded functions defined on complete metric space $(U, \hat{d});~ (\mathfrak{F}, d^*)$ is itself a complete metric space. Therefore, there exists a function $f\in \mathfrak{F}$ such that
\begin{equation*}
f_n\rightarrow f~~as~~n\rightarrow +\infty.
\end{equation*}
As $\mathfrak{D}$ is a continuous mapping, therefore, we have
\begin{eqnarray*}
\mathfrak{D}f_{n}\rightarrow \mathfrak{D}f~~as~~n\rightarrow
+\infty~~\Rightarrow f_{n+1}\rightarrow
\mathfrak{D}f~~as~~n\rightarrow +\infty.
\end{eqnarray*}
Since limit of a convergent sequence is always unique, therefore, we have $f=\mathfrak{D}f$ \textit{i.e.} $f$ is a fixed function of
$\mathfrak{D}$.
 This completes the proof.\\
\end{proof}

\begin{example}
Let $U=[0,2]$ and $\hat{d}(u,v)=|u-v|$. Let $\mathfrak{F}$ be the family of bounded functions on $[0,2]$ and $\mathfrak{D}:\mathfrak{F}\rightarrow \mathfrak{F}$ be defined as $\mathfrak{D}f=f^2$ and $d^*$ be the metric defined on $\mathfrak{f}$ as
$$
d^*(f,g)=\int^{2}_{0}|f(u)-g(u)|du.
$$
Let
\begin{eqnarray*}
f(u)=\begin{cases}1&~~u\in [0,1] \\0&~~otherwise\end{cases}
\end{eqnarray*}
\begin{eqnarray*}
g(u)=\begin{cases}-1&~~u\in [0,1] \\0&~~otherwise\end{cases}
\end{eqnarray*}
and
\begin{eqnarray*}
\alpha(f(u),g(v))=\begin{cases}2&~~u\in [0,1]
\\0&~~otherwise~.\end{cases}
\end{eqnarray*}
Clearly, $(U, \hat{d})$ is a complete metric space and $\mathfrak{D}$ is a continuous mapping. Moreover, $f$ and $g$ are
bounded functions and $\mathfrak{D}$ is $\alpha$-admissible as
\begin{equation*}
\alpha(f(u),g(v))\geq 1~~\Rightarrow~u\in [0,1]
\end{equation*}
and for $u\in [0,1]$, we have
\begin{equation*}
\alpha(\mathfrak{D}f(u),\mathfrak{D}g(v))=\alpha(f^2(u),g^2(v))\geq 1.
\end{equation*}
Now we show that $\mathfrak{D}$ is an $\alpha-\psi$ contractive mapping. To prove this, Let $\psi \in \Psi$ be a function defined as
$\psi(u)=\frac{u}{2}$.\\

\textbf{Case I}: When $u\in [0,1]$, then
\begin{eqnarray*}
\alpha(f(u),g(v))d^*(\mathfrak{D}f,\mathfrak{D}g)&=&2\int^{2}_{0}|f^2(u)-g^2(u)|du\\
&=&2\int^{1}_{0}|f^2(u)-g^2(u)|du+ 2\int^{2}_{1}|f^2(u)-g^2(u)|du \\
&=&2(0)+2(0)=0\\
&\leq&u = \psi(d^*(f,g)).
\end{eqnarray*}
\textbf{Case II}: When $u\in (1,2]$, then
\begin{equation*}
\alpha(f(u),g(v))d^*(\mathfrak{D}f,\mathfrak{D}g)=0\leq
\psi(d^*(f,g)).
\end{equation*}
Thus, all the conditions needed for Theorem \ref{t3} are fulfilled, so, there must exists a fixed function in $\mathfrak{F}$. In this example, $f$ is a fixed function of $\mathfrak{D}$.\\
\end{example}

{\bf Uniqueness}:
By considering the following hypothesis, the uniqueness of fixed function in Theorem \ref{t3} will be assured.\\
(H): for all $f,g \in \mathfrak{F}$, there exists $h \in
\mathfrak{F}$ such that
$$
\alpha(f(u),h(v))\geq 1~~and~~\alpha(g(u),h(v))\geq 1.
 $$
 \begin{theorem}
 Adding condition (H) to the hypothesis of Theorem \ref{t3}, we obtain the uniqueness of fixed function of $\mathfrak{D}$.
\end{theorem}
\begin{proof}
 Let us suppose that $f^*$ and $g^*$ be two fixed functions of $\mathfrak{D}$. From (H), there exists some $h^* \in \mathfrak{F}$ such that
\begin{equation}\label{e6}
\alpha(f^*(u),h^*(v))\geq 1~~and~~\alpha(g^*(u),h^*(v))\geq 1.
 \end{equation}
Since $\mathfrak{D}$ is $\alpha$-admissible, by (\ref{e6}), we have
\begin{equation}\label{e7}
\alpha(f^*(u),\mathfrak{D}^{n}{h^*(v)})\geq
1~~and~~\alpha(g^*(u),\mathfrak{D}^{n}h^*(v))\geq 1~~for~~all~~n \in
\mathbb{N}.
 \end{equation}
 Using (\ref{e7}) and $\alpha-\psi$ contractive condition
\begin{eqnarray*}
d^*(f^*,\mathfrak{D}^{n}h^*)&=&d^*(\mathfrak{D}f^*,\mathfrak{D}(\mathfrak{D}^{n-1}h^*))\\
&\leq&\alpha(f^*(u),\mathfrak{D}^{n-1}h^*(v))d^*(\mathfrak{D}f^*,\mathfrak{D}(\mathfrak{D}^{n-1}h^*))\\
&\leq&\psi(d^*(f^*,\mathfrak{D}^{n-1}h^*))
 \end{eqnarray*}
 which implies that
\begin{eqnarray*}
d^*(f^*,\mathfrak{D}^{n}h^*)&\leq&\psi^n(d^*(f^*,h^*))~for
~~all~~n\in \mathbb{N}.
 \end{eqnarray*}
Taking limit $n\rightarrow +\infty$, we get
\begin{equation*}
\mathfrak{D}^{n}h^*\rightarrow f^*.
 \end{equation*}
 Similarly,
 \begin{equation*}
\mathfrak{D}^{n}h^*\rightarrow g^*.
 \end{equation*}
 Uniqueness of limit gives $f^* = g^*$. This proves the theorem.\\
\end{proof}

\section{Application}
The application in this section is based on best approximation of treatment plan for tumor patients getting intensity modulated radiation therapy(IMRT).\\

 In {\cite{bort}}, Thomas Bortfeld presented some clinically complex cases to calculate dose approximation along with optimization algorithms(T. Bortfeld, \textit{Optimized planning using physical objectives and constraints}, Seminars in Radiation Oncology, 9(1999), 20-34). In {\cite{shep}},  Shepard \textit{et al.} presented some techniques such as ratio method, least-squares minimization and the maximum-likelihood estimator to develop algorithms for the problems encountered in tomotherapy(D.M. Shepard, G.H. Olivera, P.J. Reckwerdt and T.R. Mackie, \textit{Iterative approaches to dose optimization in tomotherapy}, Physics in Medicine and Biology, 45(2000), 69-90).\\

 In these techniques, a dose deposition coefficient(DDC) matrix is often computed to approximate dose distribution to each voxel in required volume of interest from every beamlet with unit intensity. But we usually get a large set of data during calculation that requires a huge computer memory and computational efficiency. As a result, small values from DDC matrix are usually truncated that affects the quality of treatment plan.\\

Fixed point iteration method is very efficient and effective technique to solve this problem. In this technique, a proper DDC matrix truncation has been used that significantly improves accuracy of results.\\

In 2013, Z. Tian \textit{et al.}{\cite{tian}} presented a fluence map optimization(FMO) model for dose calculation by splitting the DDC matrix into two components $\mathfrak{D_1}$ and $\mathfrak{D_2}$ on the basis of a threshold value. The matrix $\mathfrak{D_1}$(major component) consists those values of DDC matrix which are higher than the threshold whereas the minor component $\mathfrak{D_2}$ consists remaining values. In fact, $\mathfrak{D_1}$ represents those doses which correspond to tumor area voxels(specifically) while $\mathfrak{D_2}$ represents scatter doses passing at large distances. The problem can be interpreted as:
\begin{equation}\label{e8}
x^{(k+1)}=argmin_{x} |\mathfrak{D_1}x+\delta^{(k)}-T|
\end{equation}
\begin{equation}\label{e9}
\delta^{(k+1)}=\mathfrak{D_2} {x^{(k+1)}}
\end{equation}
The above model consists of two loops namely inner loop and outer loop. Here $k$ denotes the iteration index of outer loop. Equation (\ref{e8}) represents inner loop,  which can be solved by using iterative algorithm for value $\delta^{(k)}$. $\delta^{(k)}$ is the dose value corresponding to $\mathfrak{D_2}$. The matrix $\mathfrak{D_1}$ contains much reduced number of non zero elements as compared to DDC full matrix. So, inner loop will converge more quickly than the original matrix. The outer loop represented by equation (\ref{e9}) updates the values of $\delta^{(k+1)}$ using minor matrix $\mathfrak{D_2}$. $T$ represents the prescription dose for PTV(planned target volume) voxels and threshold dose for OAR(organs at risk) voxels. This mapping gives rise to a sequence $x^{(0)},~ x^{(1)},~ x^{(2)}, ...$ containing different dose distributions corresponding to a patient. If the matrix $\mathfrak{D_2}$ contains very small(or negligible) values and there exists such a $\lambda$ for which the contraction condition is satisfied, then the suitable dose exists for a patient at a time.\\

Following this concept, the treatment plan for more than a patient at a time, is presented through our results in a more effective way. The results proposed in this paper provide a very efficient and easy technique for estimation of suitable treatment plan.\\

 In the present case, two tumor patients have been considered with different tumor levels. Let $U$ denotes the set of all threshold intensity values(with unit Gy) to be given on particular days and in particular sessions. A patient is getting the therapy two times a day. Days and sessions are denoted by $D$ and $S$ respectively.\\
$$
U={\begin{cases}(1,D_1S_1),(\frac{1}{2},D_1S_2),(1,D_2S_1),(\frac{1}{2},D_2S_2)&~~Patient-I, \\(1,D_1S_1),(2,D_1S_2),(1,D_2S_1),(2,D_2S_2)&~~Patient-II.\end{cases}}
$$
Note that $U$ is complete being a closed and bounded subset of $\mathbb{R}^2$. Let $\mathfrak{F}=\{f_1,f_2\}$ be the family of dose functions and each function represents different dose distributions(to tumor locations) of different tumor patients during IMRT.
\begin{eqnarray*}
f_1(u)={\begin{cases}2u&~~Patient-I, \\u&~~Patient-II.\end{cases}}~~and~~~f_2(u)={\begin{cases}\frac{u}{3}&~~Patient-I, \\\frac{2u}{3}&~~Patient-II.\end{cases}}
\end{eqnarray*}
It is to be noted that $\mathfrak{F}$ is the family of bounded functions. Let $\mathfrak{D}:\mathfrak{F}\rightarrow \mathfrak{F}$ be the mapping defined as $\mathfrak{D}f=f^2-2f+2~~\forall~~f\in \mathfrak{F}$. It is required to prove that $\mathfrak{D}$ is a $\mathfrak{D}$-contraction mapping. For $u,v\in U$, we have the following cases:\\
\begin{multicols}{2}
For Patient-I\\

Case I- When $u=v=1$. Then
\begin{eqnarray*}
|\mathfrak{D}f_1-\mathfrak{D}f_2|=\frac{5}{9}~~and~~|f_1-f_2|=\frac{5}{3}.
\end{eqnarray*}

Case II- When $u=v=\frac{1}{2}$. Then
\begin{eqnarray*}
|\mathfrak{D}f_1-\mathfrak{D}f_2|=\frac{25}{36}~~and~~|f_1-f_2|=\frac{5}{6}.
\end{eqnarray*}

Case III- When $u=1,~~v=\frac{1}{2}$. Then
\begin{eqnarray*}
|\mathfrak{D}f_1-\mathfrak{D}f_2|=\frac{11}{36}~~and~~|f_1-f_2|=\frac{11}{6}.
\end{eqnarray*}

Case IV- When $u=\frac{1}{2},~~v=1$. Then
\begin{eqnarray*}
|\mathfrak{D}f_1-\mathfrak{D}f_2|=\frac{4}{9}~~and~~|f_1-f_2|=\frac{2}{3}.
\end{eqnarray*}

For Patient-II\\

Case I- When $u=v=1$. Then
\begin{eqnarray*}
|\mathfrak{D}f_1-\mathfrak{D}f_2|=\frac{1}{9}~~and~~|f_1-f_2|=\frac{1}{3}.
\end{eqnarray*}

Case II- When $u=v=2$. Then
\begin{eqnarray*}
|\mathfrak{D}f_1-\mathfrak{D}f_2|=\frac{8}{9}~~and~~|f_1-f_2|=\frac{2}{3}.
\end{eqnarray*}

Case III- When $u=1,~~v=2$. Then
\begin{eqnarray*}
|\mathfrak{D}f_1-\mathfrak{D}f_2|=\frac{1}{9}~~and~~|f_1-f_2|=\frac{1}{3}.
\end{eqnarray*}

Case IV- When $u=2,~~v=1$. Then
\begin{eqnarray*}
|\mathfrak{D}f_1-\mathfrak{D}f_2|=\frac{8}{9}~~and~~|f_1-f_2|=\frac{4}{3}.
\end{eqnarray*}
\end{multicols}

From all above cases, for Patient-I
\begin{align*}
d^*(\mathfrak{D}f_1,\mathfrak{D}f_2)&=sup \{|\mathfrak{D}f_1-\mathfrak{D}f_2|~u,v\in U\}\\
&=\frac{25}{36}\leq \frac{2}{3}\times \frac{11}{6}\\
&=\lambda d^*(f_1,f_2).
\end{align*}
 and for Patient-II
\begin{align*}
d^*(\mathfrak{D}f_1,\mathfrak{D}f_2)&=sup \{|\mathfrak{D}f_1-\mathfrak{D}f_2|~u,v\in U\}\\
&=\frac{8}{9}\leq \frac{2}{3}\times\frac{4}{3}\\
&=\lambda d^*(f_1,f_2)
\end{align*}
where $\lambda=\frac{2}{3}< 1$.\\

Thus, all the conditions required for Theorem \ref{t26} are fulfilled.
Therefore, there exists a unique fixed function $f_1$ of $\mathfrak{D}$ that yields suitable doses for two patients at the same time.

\section{Conclusion}
Till now, fixed point results have a wide range of applications in various fields such as Engineering, Functional Analysis, Optimization Theory etc. But the concept of fixed function is yet not defined. In future, this new extended concept can be applied in various forms to prove existence and uniqueness of fixed functions $\textit{i.e}$ by changing nature of mappings, using different contractive conditions, by changing space or by using different topological structures. The effectiveness of this result is directly given by the application which helps us to diagnose a number of patients at same time.


\end{document}